\documentclass[a4paper,11pt]{amsart}
\usepackage[foot]{amsaddr}
\usepackage{geometry}
\usepackage{amsfonts}
\usepackage{amsthm}
\usepackage{amssymb}
\usepackage{amsmath}
\usepackage{theoremref}
\usepackage{enumerate}
\usepackage{bbm}
\usepackage{bm}
\usepackage{mathtools}
\usepackage{a4wide}
\usepackage{xcolor}
\usepackage{comment}
\usepackage{lipsum}

\makeatletter
\newcommand{\proofpart}[2]{%
  \par
  \addvspace{\medskipamount}%
  \noindent\emph{Step #1: #2}\par\nobreak
  \addvspace{\smallskipamount}%
  \@afterheading
}
\makeatother

\DeclarePairedDelimiter\abs{\lvert}{\rvert}%
\DeclarePairedDelimiter\norm{\lVert}{\rVert}%

\makeatletter
\let\oldabs\abs
\def\abs{\@ifstar{\oldabs}{\oldabs*}}
\let\oldnorm\norm
\def\norm{\@ifstar{\oldnorm}{\oldnorm*}}
\makeatother

\makeatletter
\g@addto@macro\bfseries{\boldmath}
\makeatother

\newcommand{\Dy}{\mathcal{D}}

\newcommand{\T}{\partial\mathbb{D}}

\newcommand{\Ka}{\mathcal{K}}
\newcommand{\conj}[1]{\overline{#1}}
\newcommand{\D}{\mathbb{D}}

\renewcommand{\Dy}{\mathcal{D}}

\newtheorem{thm}{Theorem}[section]
\newtheorem{lemma}[thm]{Lemma}

\theoremstyle{definition}

\theoremstyle{definition}

\newcommand{\Addresses}{{
		\bigskip
		\footnotesize
         Linus Bergqvist, \\ \textsc{Stockholm, Sweden} \\
        \texttt{linus.lidman.bergqvist@gmail.com}

        \medskip
		
		Adem Limani, \\ \textsc{Centre for Mathematical Sciences, \\ Lund University, Sweden }\\
		\texttt{adem.limani@math.lu.se}
		
		\medskip
		Bartosz Malman, \\ \textsc{Division of Mathematics and Physics, \\ M\"alardalen University, V\"aster\aa s, Sweden} \\
		\texttt{bartosz.malman@mdu.se}

	}}

\begin{document}
\title{\textbf{Revisiting cyclic elements in growth spaces}}
\maketitle
\authors{\qquad \qquad \qquad \qquad Linus Bergqvist,  Adem Limani \& Bartosz Malman} 

\date{\today}

\begin{abstract}
We revisit the problem of characterizing cyclic elements for the shift operator in a broad class of radial growth spaces of holomorphic functions on the unit disk, focusing on functions of finite Nevanlinna characteristic. We provide results in the range of Dini regular weights, and in the regime of logarithmic integral divergence. Our proofs are largely constructive, enabling us to simplify and extend a classical result by Korenblum and Roberts, and a recent Theorem due to El-Fallah, Kellay, and Seip.

    
\end{abstract}
\section{Introduction}
\subsection{Cyclic Nevanlinna functions in growth spaces}
Let $W:(0,1] \to (0, 1]$ be a continuous non-decreasing weight (positive function) with $\lim_{t\to 0+} W(t)=0$. We denote by $A^p(W)$ the space of holomorphic functions $f$ in the unit-disc $\D$ equipped with the metric
\[
\norm{f}_{A^p(W)}= \left(\int_{\D} \abs{f(z)}^p W(1-|z|) dA(z) \right)^{\min(1,1/p)} < \infty .
\]
Since the weight $W$ is radial, it is well-known that the polynomials form a dense subset in $A^p(W)$ (for instance, see Proposition 3.1 in \cite{aleman2022backward} for a neat proof). We shall also consider the weighted growth space $A^\infty(W)$ consisting of holomorphic functions $f$ in $\D$ satisfying
\[
\lim_{\abs{z} \to 1-} W(1-|z|) \abs{f(z)} =0.
\]
Equipped with the norm
\[
\norm{f}_{A^\infty(W)} := \sup_{z\in \D} W(1-|z|) \abs{f(z)} < \infty,
\]
it becomes a separable Banach space, containing the polynomials as a dense subset. Let $\mathcal{N}$ denote the Nevanlinna class, which consists of holomorphic functions in $\D$ having finite Nevanlinna characteristic: 
\[
\sup_{0<r<1} \int_{\T} \log(1+ \abs{f(r\zeta)} ) dm(\zeta)< \infty,
\]
where $dm$ denotes the unit-normalized Lebesgue measure on the unit-circle $\T$. Given a Nevanlinna class function $f$ on $\D$, we consider the classical problem of when the set
\[
\{ f(z) z^n: \, n=0,1,2,\dots \}, \qquad z\in \D
\]
forms a dense linear span in $A^p(W)$. Such functions $f$ are said to be \emph{cyclic} in $A^p(W)$ (with respect to the shift operator $M_zf(z) = zf(z)$). Questions of this type originate back to the work of Keldysh in \cite{Keldysh} and to Beurling \cite{beurling1964critical}. Since the topologies in $A^p(W)$ induce uniform convergence on compact subsets of $\D$, cyclic functions $f$ in $A^p(W)$ can certainly not have any zeros in $\D$. The classical Nevanlinna representation allows one to express any zero-free function $f \in \mathcal{N}$ as
\[
f(z) = \exp \left( \int_{\T} \frac{\zeta +z}{\zeta -z} d\mu_f(\zeta) \right), \qquad z\in \D,
\]
where $\mu_f$ is a finite real-valued Borel measure on $\T$. In fact, a more refined Lebesgue and Jordan decomposition, in conjunction with standard properties of Poisson kernels, implies that
\begin{equation*}\label{EQ:LEBDEC}
d\mu_f = \log |f| dm + d\nu_f - d\sigma_f
\end{equation*}
where $\nu_f, \sigma_f$ are mutually singular positive finite Borel measures on $\T$, both singular wrt $dm$. We shall often refer to $\sigma_f$ as the associated negative singular part of $f$ (instead of $\mu_f$). This gives the refined Inner-outer factorization of $f$, defined by
\[
f(z) = \mathcal{O}_f(z) \Theta_{\sigma_f}(z) / \Theta_{\nu_f}(z), \qquad z\in \D,
\]
where $\mathcal{O}_f$ denotes the so-called outer factor of $f$, and $\Theta_{\sigma_f}, \Theta_{\nu_f}$ are singular inner functions. For a detailed treatment of Nevanlinna factorization and Hardy spaces, we refer the reader to the excellent book \cite{garnett}. In what follows, we shall solely restrict our attention to continuous non-decreasing weights $W$, which satisfy the following additional weak regularity condition:
\begin{equation}\label{EQ:LogwDouble}
\log \frac{1}{W(t/2)} \leq C \log \frac{1}{W(t)}, \qquad \text{for some} \, \, \, C>1. 
\end{equation}
From now and onward, we shall refer to such weights as \emph{good weights}. We will use the notation \( A \lesssim B \) to indicate that \( A \leq c B \) for some absolute constant \( c > 0 \). When both \( A \lesssim B \) and \( B \lesssim A \) hold, we simply write \( A \asymp B \). Occasionally, absolute constants may appear when carrying out estimates, but the reader should note that these constants may vary from line to line.

\subsection{Dini-regular weights}
In this section, we shall restrict our attention to weights $W$ which tend to zero in a slightly slower fashion. More precisely, we shall assume that there exists a constant $C>0$, such that
\begin{equation}\label{EQ:LogDiniReg}
    \int_{0}^x \log \frac{1}{W(t)} dt  \leq C x \log \frac{1}{W(x)}, \qquad 0<x<1.
\end{equation}
 In this regime, it turns out that there are zero-free holomorphic self-maps $f$ on $\D$ which are not cyclic in $A^p(W)$. Results of this kind were initially proved by Korenblum in \cite{korenblum1977beurling} and independently by Roberts in \cite{roberts1985cyclic}. For a certain range of weights $W$, their results assert that the cyclicity of $f$ in $A^p(W)$ is entirely contingent upon whether the associated Nevanlinna measure $\mu_f$ assigns any mass to certain exceptional $W$-sets on $\T$. Below, we clarify these points. Throughout, we let $\kappa_W$ be the associated gauge-function with respect to $W$, defined by
\[
\kappa_W(t) = t \log \frac{1}{W(t)}, \qquad 0<t<1.
\]
A compact set $K\subset \T$ of Lebesgue measure zero is said to have finite $\kappa_W$-entropy if
\[
\sum_k \kappa_W (\ell_k) <\infty.
\]
where $(\ell_k)_k$ are the lengths of the connected components $(I_k)_k$ of $\T \setminus K$. When $W(t)=t^{\alpha}$ some $\alpha>0$, such sets are typically referred to as Beurling-Carleson sets, and they play a crucial role in function theory. For instance, they precisely classify all zero sets on $\T$ of holomorphic functions in $\D$ which are smooth up to $\T$ (see \cite{taylor1970ideals}). 
We remark that the condition \eqref{EQ:LogwDouble} is equivalent to the doubling property of the gauge function $\kappa_W(t/2) \asymp \kappa_W(t)$, while \eqref{EQ:LogDiniReg} is typically referred to as $\kappa_W$ being \emph{Dini-regular}. Our main intention is to prove the following generalization of the Korenblum-Roberts Theorem.

\begin{thm}\thlabel{THM:CYCSlow} Let $0<p \leq \infty$ and $W$ be a good weight which satisfies the condition \eqref{EQ:LogDiniReg}. Then a function $f\in A^p(W) \cap \mathcal{N}$ is cyclic in $A^p(W)$ if $\sigma_f(K) = 0$ for all sets $K\subset \T$ of finite $\kappa_W$-entropy.

\end{thm}

The above theorem was initially proved by Korenblum and Roberts in the classical setting of the Bergman spaces $A^p(W)$, corresponding to weights of the form $W(t)=t^{\alpha}$. They also showed that the above condition on $\sigma_f$ is not only sufficient, but also necessary. For a wider range of weights, the same conclusion was also recently confirmed in \cite{limani2024shift}, indicating that \thref{THM:CYCSlow} is sharp. Our proof of \thref{THM:CYCSlow} is carried in the following steps. First, we simply reduce the problem to cyclicity of $f$ to the associated singular inner factor $\Theta_{\sigma_f}$. Secondly, we shall utilize a Roberts-type decomposition adapted to the corresponding weight $W$, allowing us to decompose singular measures. In the last step, our approach substantially deviates from Korenblums proof and from Roberts, where the former involves an implicit linear programming argument (see \cite{korenblum1977beurling}), while the latter invokes a quantitative version of Carleson's Corona Theorem (see \cite{roberts1985cyclic}). Instead, we shall carry out a fairly explicit construction of bounded holomorphic functions $(h_n)_n$ in $\D$, such that $\Theta_{\sigma_f} h_n-1$ have small $A^p(W)$-norms. 

\subsection{Logarithmic integral divergence}
We now restrict our attention to weights which tend to zero rapidly. That is, we assume that $\log W$ is not integrable:
    \begin{equation}\label{EQ:LogDiniDiv}
    \int_0^1 \log W(t) dt =-\infty.
    \end{equation}
Note that the above condition is slightly stronger than \eqref{EQ:LogDiniReg}, and is equivalent to the assertion that the associated gauge-function $\kappa_W(t)$ is not Dini-continuous. In this setting, our second result provides a new proof of Theorem 1.1 from \cite{el2012cyclicity} on cyclic singular inner functions in $A^p(W)$, due to El-Fallah, Kellay, and Seip. We shall re frame it in the following general form.

\begin{thm}\thlabel{THM:CYCFast} Let $0<p\leq \infty$ and $W$ be a good weight which satisfies the condition \eqref{EQ:LogDiniDiv}. Then any $f\in A^p(W) \cap \mathcal{N}$ with no zeros in $\D$ is cyclic in $A^p(W)$.
\end{thm}
A version of this Theorem was initially proved in \cite{el2012cyclicity} (See Theorem 1.1) for singular inner functions, under a different condition expressed in terms of the moment-sequence of the weight $W$. However, it is well-known that the condition therein agrees with \eqref{EQ:LogDiniDiv} for a certain class of weights $W$. The proof in \cite{el2012cyclicity}, principally relying on methods developed by Roberts in \cite{roberts1985cyclic}, is also based on a clever way of "whittling down" the measure $\mu_f$, followed by utilizing a quantitative version of Carleson's Corona Theorem. Our proof will initially follow a similar trajectory as \cite{roberts1985cyclic}, but the novelty of our work is that we shall give an explicit construction, which does not rely on the Corona Theorem.

We give a brief comparison of \thref{THM:CYCFast} with earlier works of Beurling in \cite{beurling1964critical}, and that of Nikolskii in \cite{nikolskij1976selected}. Under certain convexity assumptions on the moment sequence of $W$, Beurling proved that every bounded holomorphic function with no zeros in $\D$,
is cyclic in $\bigcup_{n\geq 1} A^2(W^n)$, equipped with the natural inductive limit topology, if and only if
\begin{equation}\label{EQ:SqDiniCond}
    \int_0^1 \sqrt{ \frac{\log \frac{1}{W(t)}}{t}} dt = + \infty.
\end{equation}
In fact, if \eqref{EQ:SqDiniCond} does not hold, then the atomic singular inner functions are not cyclic in $\bigcup_{n\geq 1} A^2(W^n)$. Beurling's original proof relied on a certain form of Bernstein approximation, which crucially required an additional convexity assumption. Later, Nikolskii established a similar result in the Hilbertian setting of $A^2(W)$ under a different log-concavity condition on the moments of $W$, which instead principally relied on methods of quasi-analyticity. It was only much later that Borichev, El-Fallah, and Hanine succeeded in removing the assumptions of Beurling and Nikolskii. They proved that atomic singular inner functions are cyclic in $A^\infty(W)$ if and only if condition \eqref{EQ:SqDiniCond} holds (see \cite{borichev2014cyclicity}). Their approach employed the so-called resolvent transform method, initially developed by Carleman, Domar, and Gelfand (see \cite{domar1975analytic} and references therein). A key component of their proof relies on \thref{THM:CYCFast} for singular inner functions, as established in \cite{el2012cyclicity}. However, their methods do not appear to extend to proving that any zero-free bounded holomorphic function is cyclic in $A^\infty(W)$ for weights more general than those considered by Beurling and Nikolskii.





\subsection{Notations and organization}
The manuscript is organized as follows. In Section \ref{SEC:2} we gather some basic preliminary lemmas in order to equip us for the following sections. The central tool therein is the simple reduction to that of cyclicity of singular inner functions. Section \ref{SEC:3} is principally concerned with the proof of \thref{THM:CYCSlow}, and principally relies on a generalized Roberts-type decomposition of singular measures. At last, Section \ref{SEC:4} is devoted to the proof of \thref{THM:CYCFast}.

\section{General properties of $A^p(W)$}\label{SEC:2}

\subsection{Compact embeddings}

Here we gather some preliminary results of $A^p(W)$-functions, which will be utilized in the later sections. We start out by recording the following observation on compact embeddings in growth spaces.

\begin{lemma}\thlabel{LEM:cptemb} The embeddings $A^p(W) \hookrightarrow A^{q}(W)$ for $0<q<p\leq \infty$, and $A^p(W) \hookrightarrow A^p(W^{s})$ for $s>1$ are compact.
\end{lemma}
\begin{proof} If $(f_n )_n$ is a sequence in the unit-ball of $A^p(W)$, then for any $p>q$ and any $0<\varepsilon<1$, we have by H\"older's inequality
\begin{multline*}
\norm{f_n}^q_{A^q(W)} \leq \int_{|z|\leq 1-\varepsilon} \abs{f_n}^q W dA + \left( \int_{1-|z|< \varepsilon} |f_n|^p W dA \right)^{q/p} \left(\int_{1-|z|< \varepsilon} W dA\right)^{1-q/p} \leq \\
\int_{|z|\leq 1-\varepsilon} \abs{f_n}^q W dA + W(\varepsilon)^{1-q/p}.
\end{multline*}
Since $(f_n)_n$ forms a normal family, Montel's Theorem implies that a subsequence $f_{n_k}$ converges uniformly on compact subsets of $\D$ to a holomorphic function $f$ in $\D$. Fatou's lemma implies that $f$ belongs to the unit-ball of $A^p(W)$ and the above estimate applied to $f_{n_k} -f$ gives
\[
\limsup_n \norm{f_{n_k}-f}^q_{A^q(W)} \leq W(\varepsilon)^{1-q/p}, \qquad \forall \varepsilon >0.
\]
This proves the first claim. For the second claim, we may repeat the same argument as before, but instead utilize the following estimate:
\[
\norm{f}^p_{A^p(W^s)} \leq \int_{|z|\leq 1-\varepsilon} \abs{f(z)}^p W(1-|z|)^s dA(z) + W(\varepsilon)^{s-1}\norm{f}^p_{A^p(W)}.
\]
The case $p=\infty$ is similar, we omit the details.
\end{proof}
\subsection{Cyclic elements in $A^p(W)$}
Here we collect two basic lemmas on cyclic elements in $A^p(W)$. We denote by $H^\infty$ the Banach space of bounded holomorphic functions in $\D$, equipped with the supremum norm 
$\norm{f}_{H^\infty} := \sup \{\abs{f(z)}: z\in \D \}$. It is not difficult to see that $H^\infty$ is the multiplier algebra of $A^p(W)$. The smallest $M_z$-invariant subspace of $A^p(W)$, which contains $f$, will be denoted by $\left[f \right]_{A^p(W)}$. With this notation, $f$ is cyclic in $A^p(W)$ if and only if $\left[f\right]_{A^p(W)} = A^p(W)$, and since polynomials are dense in $A^p(W)$, this happens if and only if $1 \in \left[f \right]_{A^p(W)}$. 

\begin{lemma}\thlabel{LEM:fH} Let $0<p\leq \infty$. Then an element $f \in A^p(W)$ is cyclic if and only if $f H^\infty := \{ fh: h\in H^\infty \} $ is dense in $A^p(W)$. 
\end{lemma}
\begin{proof} One implication is obvious. For the other it suffices to prove that $fh \in \left[f\right]_{A^p(W)}$ for any $h\in H^\infty$. To avoid redundancy, we will present the proof only for the case $p=\infty$, as the argument for $0<p<\infty$ follows in a similar manner. Since the polynomials are weak-star (sequentially) dense in $H^\infty$ (for instance, take Fej\'er means of $f$), there exists $M>0$ and polynomials $(Q_n)_n$, such that 
\begin{enumerate}
    \item[(a.)] $\sup_{n} \norm{Q_n}_{H^\infty} \leq M$,
    \item[(b.)] $Q_n \to h$ uniformly on compact subsets of $\D$.
\end{enumerate}
We now claim that $fQ_n \to fh$ in $A^\infty(W)$. Indeed, for any $0<\varepsilon<1$, we have 
\begin{multline*}
\sup_{z\in \D} \abs{f(z)\left(Q_n(z)-h(z)\right)}W(1-|z|) \leq \\ \norm{f}_{A^\infty(W)} \sup_{|z|\leq 1-\varepsilon}\abs{Q_n(z)-h(z)} + \left( M + \norm{h}_{H^\infty} \right)\sup_{1-|z|\leq \varepsilon} \abs{f(z)}W(1-|z|).
\end{multline*}
By letting $n\to \infty$ and utilizing $(b.)$, and then letting $\varepsilon \to 0+$ while using that $f\in A^\infty(W)$, the claim follows.
\end{proof}
Next, we make the following simple observation on bounded cyclic elements in $A^p(W)$.
\begin{lemma}\thlabel{LEM:bddcyc} Let $0<p\leq \infty$.  If $f \in H^\infty$ is cyclic in $A^{p}(W)$, then $f^M$ is cyclic in $A^p(W)$ for any $M>0$. 
\end{lemma}
\begin{proof} As before, we only carry out the proof in the case $p=\infty$. Note that if $f$ is cyclic in $A^\infty(W)$, then $f$ is zero-free in $\D$, and thus $f^M$ is well-defined for all $M>0$. Observe that $f= f^s f^{1-s} \in \left[f^s\right]_{A^\infty(W)}$ for all $0< s \leq 1$, hence $f^s$ is cyclic in $A^\infty(W)$ for all $0 < s\leq 1$. Now if $Q_n$ polynomials such that $Q_n f \to 1$ in $A^\infty(W)$, then multiplying by the bounded function $f^s$, we get $f^{s} \in \left[ f^{1+s} \right]_{A^\infty(W)}$, which by the previous argument implies that $f^{M}$ cyclic in $A^\infty(W)$ for all $0< M \leq 2$. By means of induction, we may iterate the above argument to deduce that $f^M$ is cyclic in $A^\infty(W)$ for all $M>0$.  \qedhere

\end{proof}

At last, we make one more simple observation, which allows us reduce our problems to that of characterizing cyclic singular inner functions in $A^p(W)$.

\begin{thm}\thlabel{LEM:CycInner} Let $0<p\leq \infty$ and $f\in \mathcal{N} \cap A^\infty(W)$ with Nevanlinna factorization $f= \mathcal{O} \Theta_{\mu}/\Theta_{\nu}$, where $\mu, \nu$ are mutually singular positive measures. If $\Theta_{\mu}$ is cyclic in $A^p(W)$, then $f$ is also cyclic in $A^p(W)$.
\end{thm}
\begin{proof}
It is a standard fact that one can express $f= a/b$, where $a,b \in H^\infty$ and $\Theta_{\mu}$ is the inner factor of $a$. Recall that bounded outer functions are weak-star (sequentially) cyclic in $H^\infty$ (for instance, see Theorem 7.4 in \cite{garnett}), and hence they can easily be shown to be cyclic in $A^\infty(W)$ by following an argument similar to the proof of \thref{LEM:fH}. This implies that $\Theta_{\mu} \in \left[a\right]_{A^p(W)} \subseteq \left[ f \right]_{A^p(W)}$. The claim now follows.
\end{proof}

\section{Dini-regular weights}\label{SEC:3}

\subsection{Cyclic inner functions}
Our main goal in this section is to prove the following theorem on cyclic inner functions. 

\begin{thm}\thlabel{THM:CIDR} Let $W$ be a good weight which satisfies the condition \eqref{EQ:LogDiniReg}. Then $\Theta_{\mu}$ is cyclic in $A^\infty(W)$ if $\mu(K)=0$ for any set $K\subset \T$ of finite $\kappa_w$-entropy.
\end{thm} 

We obtain \thref{THM:CYCSlow} as an immediate corollary of \thref{THM:CIDR} in conjunction with \thref{LEM:CycInner}.

\subsection{A Roberts decomposition} 
Our principal tool in this subsection will be a Roberts-type decomposition, adapted to weights $W$ for which the associated gauge function $\kappa_W(t) \to 0$ as $t\to 0+$. We denote by $\Dy_n$ a collection of $2^n$ disjoint dyadic arcs of length $2\pi 2^{-n}$ which partition $\T$. Given a weight $W$, we declare that a sequence of positive integers $\{n_k\}_{k=0}^\infty$ gives rise to a $W$-adapted dyadic grid $\cup_{k=0}^\infty \Dy_{n_k}$ if there exists a constant $\gamma >0$, such that 
\begin{equation}\label{EQ:Wgrid}
    \sup_{k\geq 0} \frac{W(2^{-n_k})^{\gamma}}{W(2^{-n_{k+1}})} < \infty.
\end{equation}

We shall derive a natural generalization of the Roberts decomposition in \cite{roberts1985cyclic}. Notably, similar decompositions have also appeared in \cite{ivrii2024beurling}, \cite{limani2024shift} and in \cite{bourhim2004boundary}, but ours is essentially as general as it gets. 

\begin{thm}\thlabel{THM:ROBDEC}[Roberts decomposition] Let $\mu$ be a positive finite Borel measure on $\T$ which is singular with respect to $dm$, and let $W$ be a continuous non-decreasing weight with $\lim_{t\to 0+} \kappa_W(t)=0$. Then for any integer $n_0>0$ any $\eta >0$, and any $W$-adapted dyadic grid $\cup_{k=0}^\infty \Dy_{n_k}$, there exist positive finite Borel measures $(\mu_k )_k$ and $\mu_\infty$ on $\T$ which decompose $\mu$ as
\[
\mu = \mu_\infty + \sum_{k=0}^\infty \mu_k,
\]
and where the pieces satisfy the following:

\begin{enumerate}
   \item[(i)]
    \[
    \sup_{|I|\leq 2^{-n_k} } \mu_k(I) \leq \eta \kappa_W ( 2^{-n_k}), \qquad k=0,1,2, \dots
    \]
    where the supremum is taken over all arcs $I\subset \T$ of length at most $2^{-n_k}$.
    \item[(ii)] $\mu_\infty$ is supported on a set of finite $\kappa_W$-entropy.

\end{enumerate}
    Furthermore, if $\mu$ does not assign mass to any set of finite $\kappa_W$-entropy, then the above decomposition holds with $\mu_\infty \equiv 0$, for any choice of parameters $n_0, \eta$ and $W$-adapted dyadic grid $\cup_{k=0}^\infty \Dy_{n_k}$.
\end{thm}

\begin{proof}[Sketch of proof] We run the argument as in Roberts paper, utilizing \thref{LEM:Wadapgrid} (proved below). This gives the decomposition 
\[
\mu = \mu_\infty + \sum_{k\geq 0} \mu_k,
\]
where $\mu_\infty$ is supported on the set $H := \cap_{k=0}^\infty H_k$, where each $H_k$ is the union of so-called \emph{heavy arcs} $I\in \Dy_{n_k}$, satisfying 
\[
\mu_k(I) = \eta \kappa_W(|I|).
\]
We first observe that 
\[
\abs{H_k} = \sum_{I\in \Dy_{n_k} \, \text{heavy}} \abs{I} \leq \left(\eta \log \frac{1}{W(2^{-n_k})}\right)^{-1} \sum_{I\in \Dy_{n_k} \, \text{heavy}} \mu(I) \to 0, \qquad k\to \infty,
\]
hence $H$ has zero Lebesgue measure. Let $L_k$ denote the set of interiors of the arcs in $\Dy_{n_k}$, which are not heavy, but intersect $H_{k-1}$. Then $H'= \T \setminus \cup_k \cup_{\ell \in L_k} \ell$ is a compact set, which contains $H$, and differs from it only on a countable set. It therefore suffices to verify that $H'$ has finite $\kappa_W$-entropy. To this end, we note that
\[
\sum_k\sum_{\ell \in L_k} \kappa_W(\ell) = \frac{1}{\eta}\sum_k |L_k| \cdot \log \frac{1}{W(2^{-n_k})} \leq \sum_k \abs{H_{k-1}} \log \frac{1}{W(2^{-n_k})}.
\]
The $W$-adapted grid assumption in \eqref{EQ:Wgrid} ensures that 
\[
\log \frac{1}{W(2^{-n_k})} \lesssim \log \frac{1}{W(2^{-n_{k-1}})}, \qquad k=1,2,\dots.
\]
With this at hand, we deduce that
\[
\sum_k\sum_{\ell \in L_k} \kappa_W(\ell) \lesssim \sum_k \abs{H_{k}} \log \frac{1}{W(2^{-n_k})} 
\leq \sum_k \frac{1}{\eta} \mu_k(I) \leq \frac{1}{\eta} \mu(\T).
\]
This shows that $H'$ has finite $\kappa_W$-entropy, hence the claim on the support of $\mu_\infty$ follows.
\end{proof}

We will later use the measures $\mu_k$ from the Roberts decomposition to explicitly construct functions $F_n \in H^\infty$ such that $F_n \Theta_\mu \rightarrow 1$ in $A^\infty(W)$.

But first, we shall need a lemma, which previously appeared in \cite{limani2024shift} (see Lemma 2.3), allowing us to selected a $W$-adapted dyadic grid $W$ with some additional property, that will be crucial in proving \thref{THM:CIDR}. Here, we shall make use of the condition \eqref{EQ:LogwDouble} on $W$ being good.

\begin{lemma}\thlabel{LEM:Wadapgrid} Let $W$ be a good weight. Then for any integer $n_0>0$, there exists a sequence of positive integers $(n_k)_{k=0}^\infty$ which gives rise to a $W$-adapted dyadic grid $\cup_{k=0}^\infty \Dy_{n_k}$, and satisfies the additional condition:
\begin{equation}\label{EQ:ProdDouble}
W(2^{-n_{k+1}}) \leq \prod_{j=0}^k W(2^{-n_j}), \qquad k=0,1,2,\dots
\end{equation}

\end{lemma}

\begin{proof}
For the sake of abbreviation, we set $w(t) = \log \frac{1}{W(t)}$ and note that $w$ is non-increasing with $w(t) \uparrow \infty$ as $t \downarrow 0$. According to \eqref{EQ:LogwDouble}, there exists a constant $C=C(W)>1$, such that 
\begin{equation} \label{EQ:wDoubling}
    w(t/2) \leq C w(t), \qquad 0<t<1.
\end{equation}
By means of induction, assume that $n_0<n_1<\dots<n_k$ has been constructed, and pick $0<\delta_k < 2^{-n_k}$ such that 
\[
10 \leq \frac{w(\delta_k)}{w(2^{-n_k})} \leq 10 \cdot C^{9}.
\]
Now choose $n_{k+1} > n_k$ be the unique integer such that $2^{-n_{k+1}} \leq \delta_{k} < 2^{1-n_{k+1}}$, then we again obtain from \eqref{EQ:wDoubling} that 
\[
10 \leq  \frac{w(2^{-n_{k+1}})}{w(2^{-n_k})} \leq 10 \cdot C^{10}.
\]
This construction provides a sequence of positive integers $(n_k)_{k=0}^\infty$, which give rise to a $W$-adapted grid since 
\[
\frac{1}{W(2^{-n_{k+1}})} = \exp  w(2^{-n_{k+1}}) \leq \exp 10\cdot C^{10} w(2^{-n_k}) = \left( \frac{1}{W(2^{-n_k})} \right)^{10 \cdot C^{10} }.
\]
In order to verify that \eqref{EQ:ProdDouble} holds, we observe that an iteration gives 
\[
w(2^{-n_j}) \leq 10^{-1} w(2^{-n_{j+1}})\leq \dots \leq \left(10^{-1} \right)^{k-j +1} w(2^{-n_{k+1}}), \qquad j=0,1, \dots, k.
\]
This implies that
\[
\sum_{j=0}^k w(2^{-n_j}) \leq  w(2^{-n_{k+1}}) \sum_{j=0}^k \left(10^{-1} \right)^{k-j +1} \leq \frac{1}{9}w(2^{-n_{k+1}}).
\]
Expressing this in terms of $W$, conclude that \eqref{EQ:ProdDouble} holds. \qedhere

\end{proof}

\subsection{The main construction}
 Let $W$ be a good weight which satisfies the condition \eqref{EQ:LogDiniReg} and let $\mu$ be a positive finite singular measure with the property that
\[
\mu(K)=0
\]
for any set $K \subset \T$ of finite $\kappa_W$-entropy. For any $n_0>0$ and any $\eta>0$, we may apply \thref{LEM:Wadapgrid} in conjunction with the Roberts decomposition of $\mu$ to find a sequence of positive integers $(n_k)_k$ and positive measures $(\mu_k)_k$ such that the following holds:
\begin{enumerate}
    \item[(i)] $\mu_k(I) \leq \eta \kappa_W (2^{-n_k})$ for any arc $I \subset \T$ of length $|I|\leq 2^{-n_k}$ for $k=0,1,2,\dots$.
    \item[(ii)] $\mu_f = \sum_k \mu_k$.
    \item[(iii)] There exists a large number $\gamma>1$, such that 
    \[
    \sup_{k} \frac{W^{\gamma}(2^{-n_k})}{W(2^{-n_{k+1}})} < \infty.
    \]
\end{enumerate}
For each $k\geq 0$, we define non-negative functions by 
\begin{equation}\label{DEF:fk}
    f_k(\zeta) := \sum_{I \in \Dy_{n_k}} \frac{\mu_k(I)}{|I|}1_I(\zeta), \qquad \zeta \in \T, \qquad k=0,1,2, \dots
\end{equation}
We will ultimately consider the functions of the form
$$
F_{\eta} := \exp \left( H \left( \sum_{k=0}^\infty f_k \right) \right),
$$
where $H$ denotes the Herglotz transform, and $\eta$ is the parameter from the Roberts decomposition. We then show that as $\eta \rightarrow 0$, the sequence $F_\eta \Theta_\mu$ is uniformly bounded and tends to $1$ pointwise in $\D$, which implies convergence (of a subsequence) to $1$ in $A^\infty(W)$, thus proving that $\Theta_\mu$ is cyclic.

To this end, we will study $|F_\eta \Theta_\mu|$, and therefore the Poisson integrals of the real-valued measures $\nu_k$ on $\T$ defined by
\[
d\nu_k = f_k dm - d\mu_k,  \qquad \zeta \in \T, \qquad k=0,1,2, \dots
\]
We start with the following simple lemma.
\begin{lemma}\thlabel{LEM:nukest} For each $k$ and for any arc $I \subset \T$, we have the estimate
\[
\abs{\nu_k(I)} \leq 4\eta \kappa_W(2^{-n_k}).
\]
\end{lemma}
\begin{proof} From the construction of $f_k$ it follows that $\nu_k(I)=0$ for any dyadic arc $I \in \Dy_{n_k}$. Now for an arbitrary arc $I \subset \T$, we can decompose $I$ as a union of disjoint intervals in $\Dy_{n_k}$ together with the intersection of $I$ with the at most two dyadic arcs $I_1, I_2 \in \Dy_{n_k}$ which contain the end-points of $I$. This implies that 
\[
\abs{\nu_k(I)} \leq \int_{I \cap I_1} f_k dm + \int_{I \cap I_2} f_k dm + \mu_k(I_1) + \mu_2(I_2) \leq 4 \eta \kappa_W(2^{-n_k}).
\] 
    
\end{proof}
We denote the Poisson extension of a measure $\nu$ on $\T$ by 
\[
P(\nu)(z) := \int_{\T} \frac{1-|z|^2}{|\zeta-z|^2}d\nu(\zeta), \qquad z\in \D.
\]
Our next lemma allows us to transform estimates of $\nu_k$ to growth estimates on their Poisson extension $P(\nu_k)$.
\begin{lemma}\thlabel{LEM:Poisest} There exists an absolute constant $C>0$, such that 
\[
P(\nu_k)(z) \leq C \eta \kappa_W(2^{-n_k}) \min \left( 2^{n_k}, \frac{1}{1-|z|} \right), \qquad z\in \D.
\]
\end{lemma}
\begin{proof}
We primarily note that since $\nu_k(\T)=0$, an integration by parts gives 
\[
P(\nu_k)(z) = \int_{0}^{2\pi} \frac{1-|z|^2}{\abs{e^{it}-z}^2}d\nu_k(e^{it}) = (1-|z|^2) \int_{0}^{2\pi} \nu_k(I(e^{it}) ) \frac{d}{dt} \abs{e^{it}-z}^{-2} dt,\qquad z\in \D,
\]
where $I(e^{it})$ smallest closed arc, connecting $1$ to $e^{it}$. We make the following two observations. First, it is straightforward to verify that there exists a numerical constant $C>0$, such that 
\[
\abs{\frac{d}{dt} \abs{e^{it}-z}^{-2}} \leq C\abs{e^{it}-z}^{-3}, \qquad z\in \D.
\]
Applying this observation in conjunction with \thref{LEM:nukest}, we get
\[
\abs{P(\nu_k)(z)} \leq C (1-|z|^2) \int_{0}^{2\pi} \eta \kappa_W(2^{-n_k}) \frac{dt}{\abs{e^{it}-z}^3} \leq 4C\eta \kappa_W(2^{-n_k}) \frac{1}{1-|z|}, \qquad z\in \D.
\]
In the last step utilized standard Poisson estimates, for instance, see Theorem 1.7 in \cite{hedenmalmbergmanspaces}. On the other hand, the definition of $\nu_k$ and $\mu_k$ yields the estimate
\[
P(\nu_k)(z) \leq P(f_k)(z) \leq \sum_{I\in \Dy_{n_k}} \eta 2^{n_k} \kappa_W(2^{-n_k}) P(1_I)(z) = \eta 2^{n_k} \kappa_W(2^{-n_k}), \qquad z\in \D.
\]
This completes the proof. 
\end{proof}
We are now ready to carry out the proof of the main result. 
\begin{proof}[Proof of \thref{THM:CIDR}] 
Let $\mu$ be a positive finite Borel measure on $\T$, which is singular with respect to $dm$ on $\T$, with the property that 
\[
\mu(E)=0
\]
for any subset $E \subset \T$ of finite $\kappa_W$-entropy, and let $\Theta = \Theta_\mu$ denote the associated singular inner factor. To avoid redundancy, we make two simple observations. Note that if we prove the theorem for $A^\infty(W)$, then it also holds for $A^\infty(W^\gamma)$ since the $\kappa_W$-entropy condition is invariant under power transformations of the weight $W$. Furthermore, by Lemma \ref{LEM:cptemb} $A^\infty(W^{1/p}) \hookrightarrow A^p(W)$, and so it suffices to only carry out the proof for $A^\infty(W)$. Fix $n_0>0$ (this parameter will not play a role) and let $\eta >0$, to be specified later. According to \thref{LEM:Wadapgrid}, there exists positive integers $(n_k)_{k=0}^\infty$ which give rise to an $W$-adapted dyadic grid $\cup_{k=0}^\infty \Dy_{n_k}$. We now invoke Roberts decomposition with the above parameters. 
\proofpart{1}{Estimates in lacunary discs:} Note that an application of \thref{LEM:Poisest} implies that
 
\[
\sum_j P(\nu_j)(z) \leq C \eta \sum_{0\leq j \leq k} 2^{n_j} \kappa_W(2^{-n_j}) + \frac{C\eta }{1-|z|} \sum_{j>k} \kappa_W(2^{-n_j}), \qquad 1-|z| = 2^{-n_k}. 
\]
Note that the first term can be estimated using the assumption \eqref{EQ:ProdDouble} of \thref{LEM:Wadapgrid}: 
\[
\sum_{0\leq j \leq k} 2^{n_j} \kappa_W(2^{-n_j}) = \sum_{0\leq j \leq k} \log \frac{1}{W(2^{-n_j})} \leq \log \frac{1}{W(2^{-n_{k+1}})}, \qquad k=0,1,2,\dots
\]
While for the second term, we now utilize the Dini-regularity condition \eqref{EQ:LogDiniReg} of the weight $W$, which implies
\[
\sum_{j>k} \kappa_W(2^{-n_j}) \lesssim \sum_{j>k} \int_{2^{-n_{j+1}}}^{2^{-n_j}} \log \frac{1}{W(t)} dt = \int_{0}^{2^{-n_{k+1}}} \log \frac{1}{W(t)} dt \lesssim \kappa_W(2^{-n_{k+1}}), \qquad k=0,1,2,\dots
\]
Invoking harmonicity and the maximum principle, we actually get that 
\begin{equation}\label{EQ:Anest}
\sum_j P(\nu_j)(z) \leq C\eta 2^{n_{k+1}} \kappa_W(2^{-n_{k+1}}) = C \eta \log \frac{1}{W(2^{-n_{k+1}})}, \qquad |z| \leq 1-2^{-n_{k}}.
\end{equation}
\proofpart{2}{Uniformly bounded growth:} Fix a large integer $N>0$ and consider the bounded outer functions
\[
F_{N} := \exp \left( H \left( \sum_{k=0}^N f_k \right) \right) =  \exp \left( \sum_{k=0}^N \sum_{I \in \Dy_{n_k}} \frac{\mu_k(I)}{|I|} H (1_I) \right), 
\]
where $H$ denotes the Herglotz transform. Here the truncation by $N>0$ is just to ensure that the $F_N$'s are bounded, and note also that the $F_{N}$'s also depend on $\eta$, the parameter in the precise Roberts decomposition of $\mu$. Let $\Theta$ be the singular inner function with associated singular measure $\mu$, and fix $\rho>0$. We claim that there exists a constant $C>0$, independent of $N$ and $\eta >0$, such that
\begin{equation} \label{EQ:FNest}
\sup_{z\in \D} \, \abs{F_N(z) \Theta(z)} W(1-|z|)^{\rho} \leq C.
\end{equation}
To this end, note that
\[
\abs{F_N (z)\Theta(z)} \leq  \exp \left( \sum_{0\leq k \leq N } P(\nu_k)(z) \right) , \qquad z\in \D.
\]
Now on each annuli $R_k := \{2^{-n_{k+1}} < 1-|z| \leq 2^{-n_k} \}$, we have, according to \eqref{EQ:Anest}, the following estimate:
\[
\sup_{z\in R_k} \exp \left( \sum_{0\leq k \leq N } P(\nu_k)(z) \right) W(1-|z|)^{\rho} \leq  \frac{W(2^{-n_k})^{\rho}}{W(2^{-n_{k+1}})^{\eta C}} \leq C'.
\]
where in the last line we utilize that the assumption that $(n_k)_k$ gives rise to a $W$-adapted dyadic grid, which ensures that $C'>0$ does not depend on $k$. This holds whenever the parameter $\eta>0$ is sufficiently small, since the constant $C>0$ is universal. On the other hand, inside that disc $|z| \leq 1-2^{-n_0}$, we have 
\begin{multline*}
\sup_{|z| \leq 1-2^{-n_0}} \abs{F_N(z) \Theta(z)} W(1-|z|)^{\rho} \leq \sup_{|z| \leq 1-2^{-n_0}} \exp \left(\sum_{0\leq k \leq N } P(\nu_k)(z) \right) W(1-|z|)^{\rho}   \\ \leq
\sup_{|z| \leq 1-2^{-n_0}} \exp \left( \frac{\eta C}{1-|z|} \sum_{0\leq k \leq N} \kappa_W(2^{-n_k}) \right) = \exp \left( \eta 2^{n_0} C \sum_{0\leq k \leq N} \kappa_W(2^{-n_k}) \right)  . 
\end{multline*}
Here we estimated $W$ by a constant, since it is not decaying inside the disc $|z|\leq 1-2^{-n_0}$. Invoking the Dini-regularity assumption on $\kappa_W$ in \eqref{EQ:LogDiniReg} once again (here we actually only need the logarithmic integrability of $W$), we find that 
\[
\sum_{k=0}^\infty \kappa_W(2^{-n_k})\lesssim  \sum_{k=0}^\infty \int_{2^{-n_{k+1}}}^{2^{-n_k}} \log \frac{1}{W(t)} dt \leq \int_{0}^{1} \log \frac{1}{W(t)} dt \leq c,
\]
where $c>0$ is a constant only depending on $W$. Consequently, we obtain  
\[
\sup_{|z| \leq 1-2^{-n_0}} \abs{F_N(z) \Theta(z)} W(1-|z|)^{\rho} dA(z) \lesssim \exp ( \eta 2^{n_0} Cc ) .
\]
This proves \eqref{EQ:FNest}. Letting $N \to \infty$, we arrive at 
\[
\sup_{z\in \D} \abs{F_{\eta}(z)\Theta(z)}W(1-|z|)^{\rho} \leq C,
\]
where $C>0$ is independent of $0<\eta<1$ small enough, and 
\[
F_{\eta}(z) = \exp \left( \sum_{k=0}^\infty \sum_{I \in \Dy_{n_k}} \frac{\mu_k(I)}{|I|} H (1_I)(z) \right), \qquad z\in \D.
\]

\proofpart{3}{Convergence in norm:}
In order to complete the proof, we shall need one more lemma:

\begin{lemma}\thlabel{LEM:CONVFeta} The measure 
$\lambda_{\eta} := \sum_{k=0}^\infty \sum_{I \in \Dy_{n_k}} \frac{\mu_k(I)}{|I|} 1_I$ converges to $\mu$ weak-star in the topology of measures $M(\T)$ as $\eta \to 0$.
    
\end{lemma}
\begin{proof}
Fix a number $\varepsilon>0$ and a continuous function $\psi$ on $\T$. Then for $M>1$ large enough, uniform continuity of $\psi$ ensures that 
\begin{equation}\label{EQ:uniest1}
\sup_{|\zeta-\xi|\leq 1/n_k } \abs{\psi(\zeta)-\psi(\xi)} \leq \varepsilon, \qquad n_k >M.
\end{equation}
Recall that $\nu_k := \sum_{I \in \Dy_{n_k}} \left( \frac{\mu_k(I)}{|I|} 1_I - \mu_k\right)$ and note that
\[
\abs{\int_{\T} \psi d\nu_k } \leq \sum_{I\in \Dy_{n_k}} \abs{ \frac{\mu_k(I)}{|I|}\int_I \psi dm - \int_I \psi d\mu_k}.
\]
If $\xi_I$ denotes the center of each arc $I \in \Dy_{n_k}$, then for $n_k>M$:
\[
\abs{ \frac{\mu_k(I)}{|I|}\int_I \psi dm - \int_I \psi d\mu_k} \leq \frac{\mu_k(I) }{|I|} \int_I \abs{\psi- \psi(\xi_I)} dm + \int_I \abs{\psi-\psi(\xi_I)} d\mu_k \leq 2 \varepsilon \mu_k(I).
\]
Summing over all $I\in \Dy_{n_k}$, we get 
\[
\abs{\int_{\T} \psi d\nu_k } \leq \sum_{I \in \Dy_{n_k}} 2 \varepsilon \mu_k(I) \leq 2 \varepsilon \mu_k(\T), \qquad n_k >M.
\]
Now recall that the estimate $\mu_k$ on each $I \in \Dy_{n_k}$ implies that
\[
\mu_k(\T) = \sum_{I \in \Dy_{n_k}} \mu_k(I) \leq \eta 2^{n_k} \kappa_W(2^{-n_k})= \eta \log \frac{1}{W(2^{-n_k})}, \qquad k=0,1,2,3, \dots
\]
From this, it follows that
\[
\abs{\int_{\T} \psi d\nu_k } \leq \norm{\psi}_\infty \norm{\nu_k} \leq 2\norm{\psi}_\infty \mu_k(\T) \leq 2\eta  \norm{\psi}_\infty \log \frac{1}{W(2^{-n_k})}.
\]
With this at hand, we may write 
\[
\nu_{\eta}:= \lambda_{\eta} - \mu = \sum_{k} \nu_k.
\]
Applying these estimates in conjunction with \eqref{EQ:uniest1}, we get
\begin{multline*}
\abs{\int_{\T} \psi d\nu_{\eta} } \leq \sum_{k \text{: } n_k \leq M}  \abs{\int_{\T} \psi d\nu_k} + \sum_{k \text{: } n_k > M}  \abs{\int_{\T} \psi d\nu_k} \\
\leq  2\eta \norm{\psi}_{\infty} \sum_{k \text{: } n_k \leq M}\log \frac{1}{W(2^{-n_k})}  + 2\varepsilon \mu(\T).
\end{multline*}
Letting $\eta \to 0$ finishes the proof. \qedhere

\end{proof}
Finally, observe that $A^\infty(W^{\rho})$ is compactly contained in $A^\infty(W)$ whenever $0<\rho<1$, hence we can find a subsequence $(\eta_n)_n$ tending to zero, such that $F_{\eta_n}\Theta$ converges in $A^\infty(W)$. However, since $F_{\eta_n} \Theta \to 1$ pointwise in $\D$ by \thref{LEM:CONVFeta}, we conclude that 
\[
\lim_n \sup_{z\in \D} W(1-|z|) \abs{F_{\eta_n}(z)\Theta(z)-1}=0.
\]
This completes the proof of the theorem. \qedhere

\end{proof}

%
%
%
%

\section{Logarithmic integral divergence}\label{SEC:4}

\subsection{Cyclicity of inner functions}
As in the previous subsection, \thref{LEM:CycInner} allows us to reduce \thref{THM:CYCFast} to proving the following result, which is the main purpose of this section.

\begin{thm}\thlabel{THM:CIDD} Let $W$ be a good weight satisfying the condition \eqref{EQ:LogDiniDiv}. Then any singular inner function $\Theta_{\mu}$ is cyclic in $A^\infty(W)$.
\end{thm}

\subsection{Reformulating the logarithmic integral divergence}
Here we gather the main technical lemmas required to prove \thref{THM:CYCFast}. Our first lemma is essentially a discretized reformulation of condition \eqref{EQ:LogDiniDiv}, inspired by Lemma 2.1 in \cite{el2012cyclicity}.

\begin{lemma}\thlabel{LEM:DISDINI} Let $W$ be a good weight. Then $W$ satisfies the condition \eqref{EQ:LogDiniDiv} if and only if, for any $A>1$, there exists positive integers $(n_k)_k$ which satisfy the following conditions:
\begin{enumerate}
    \item[(i)] $n_{k+1}\geq A n_k$, $k=0,1,2,\dots$,
    \item[(ii)] $\sum_{j=0}^k \log \frac{1}{W(1/n_j)} \leq \log \frac{1}{W(1/n_k)} $, for $k=0,1,2,\dots$,
    \item[(iii)] $\sum_j \frac{1}{n_{j+1}} \log \frac{1}{W(1/n_j)} = +\infty$.
\end{enumerate}
\end{lemma}
\begin{proof}
    Suppose that $W$ satisfies the condition \eqref{EQ:LogDiniDiv}. Pick an arbitrary integer $m_0>0$, and inductively choose $m_{k+1}>m_k$ to be the smallest integer for which 
    \[
    \log \frac{1}{W(1/m_{k+1})} \geq 2\log \frac{1}{W(1/m_k)}, \qquad k=0,1,2,\dots.
    \]
    With the sequence $(m_k)_k$ at hand, we note that for each $k\geq 1$, we have
    \begin{multline*}
    \int_{1/m_{k+1}}^{1/m_k} \log \frac{1}{W(t)} dt \asymp \left( \frac{1}{m_k} - \frac{1}{(m_{k+1}-1)} \right) \log \frac{1}{W(1/m_k)} + \int_{1/m_{k+1}}^{1/(m_{k+1}-1)} \log \frac{1}{W(t)} dt  \\ 
    \asymp \left( \frac{1}{m_k} - \frac{1}{m_{k+1}}\right)\log \frac{1}{W(1/m_k)},
    \end{multline*}
where the assumption \eqref{EQ:LogwDouble} was utilized in the last step. Hence the condition \eqref{EQ:LogDiniDiv} translates into
    \[
     \int_0^{1/m_1} \log \frac{1}{W(t)} dt \asymp \sum_{k=1}^\infty \left( \frac{1}{m_k} - \frac{1}{m_{k+1}}\right) \log \frac{1}{W(1/m_k)} = + \infty.
     \]
Using the assumption \eqref{EQ:LogwDouble} and the definition of $(m_k)$, we conclude that
\begin{equation}\label{EQ:logDiscrete}
\sum_k \frac{1}{m_{k+1}} \log \frac{1}{W(1/m_k)} = +\infty.
\end{equation}
Condition \eqref{EQ:LogDiniDiv} ensures that the sequence $(m_k)$ satisfies the properties $(ii)-(iii)$, hence we only need to modify it to meet $(i)$. To this end, fix $A>1$ and observe that for each $k\geq 1$,
    \begin{multline*}
    \sum_{m_k/ A \leq m_{j} \leq m_k} \frac{1}{m_{j}} \log \frac{1}{W(1/m_j)} \leq \frac{A}{m_k} \sum_{m_k/ A \leq m_j \leq m_k} \log \frac{1}{W(1/m_j)} \leq  \\ 
    A\sum_{j\geq 0} 2^{-j} \frac{1}{m_k} \log \frac{1}{W(1/m_k)}  \lesssim \frac{1}{m_k} \log \frac{1}{W(1/m_k)}.
    \end{multline*}
    Again, a similar argument as when \eqref{EQ:logDiscrete} was deduced, shows that we may drop all the integers $m_k$ which violate $(i)$, while still maintaining the condition $(iii)$. The remaining part of $(m_k)_k$ may then be re-labeled as $(n_k)_k$. The converse easily follows from \eqref{EQ:logDiscrete}, where the inequality can now be reversed by the assumptions on $(n_k)$.
\end{proof}

The following lemma will play a crucial role in or developments, and is essentially a linear programming problem that can be solved explicitly. 

\begin{lemma} \thlabel{LEM:LINPROG}
Let $W$ be a good weight which satisfies the condition \eqref{EQ:LogDiniDiv}, and let $(n_k)_{k=0}^\infty$ be positive integers which fulfill the hypothesis of \thref{LEM:DISDINI}. Then for any $0<\varepsilon_0<1$ and any integer $N>1$, there exists a constant $c_0>0$, independent of $\varepsilon_0$ and $N$, and positive numbers $\varepsilon_1, \varepsilon_2, \dots, \varepsilon_N$, such that the following statements hold:
\begin{enumerate}
    \item[(i)]
    $\sum_{j=0}^N \varepsilon_j =1$,
    \item[(ii)]$\sum_{0\leq j \leq k} n_j \varepsilon_j \leq  c_0\log \frac{1}{W(1/n_k)}, \qquad k=0,1,2,\dots,N$,
    \item[(iii)] $\sum_{k<j\leq N} \varepsilon_j \leq \frac{c_0}{n_{k+1}} \log \frac{1}{W(1/n_k)}, \qquad k=0,1,2,\dots,N$.
\end{enumerate}
    
\end{lemma}
Note that $(i)$ in conjunction with $(iii)$ is only possible if \eqref{EQ:LogDiniDiv} holds, which is visible from \thref{LEM:DISDINI}.
\begin{proof}
Note that we may assume that $\lim_{t\downarrow 0+} t \log \frac{1}{W(t)}=0$, otherwise the task becomes simple. Fix $0<\varepsilon_0<1$, and let $(n_k)_{k=0}^\infty$ be a sequence satisfying the hypothesis of \thref{LEM:DISDINI}. For brevity, we may set $w_k := \frac{1}{n_{k+1}} \log \frac{1}{W(1/n_k)}$ and take 
\[
\varepsilon_k := c_0(w_k - w_{k+1}), \qquad k=1,2,\dots,N-1, \qquad \varepsilon_N = c_0 w_{N}
\]
where $c_0>0$ such that $\varepsilon_0 + c_0 w_1 = 1$. This implies that $(i)$ holds. The assumption of the $(n_k)_{k=1}^\infty$, readily implies $(ii)$: 
\[
\sum_{j=0}^k  n_j \varepsilon_j \leq c_0 \sum_{j=0}^k \log \frac{1}{W(1/n_j)} \leq c_0 \log \frac{1}{W(1/n_k)}, \qquad k=0,1,2,\dots, N.
\]
The verification of $(iii)$ is also simple:
\[
\sum_{k<j\leq N} \varepsilon_j = \frac{c_0}{n_{k+1}} \log \frac{1}{W(1/n_k)}, \qquad k=0,1,2,\dots,N.
\]
By means of increasing $c_0>0$ slightly, we can also ensure it to be independent of $\varepsilon_0>0$. \qedhere

\end{proof}

\subsection{Decomposing singular measures and Poisson estimates}

Let $\Theta:= \Theta_{\mu}$ be any singular inner function with associated singular measure $\mu$, which we for simplicity shall assume satisfies $\mu(\T)\leq 1$. Fix an arbitrary $0<\varepsilon_0<1$, a positive integer $N>0$, and let $(n_k)_{k=1}^\infty$ be positive integers satisfying the hypothesis of \thref{LEM:DISDINI}. According to \thref{LEM:LINPROG}, there exists a constant $c_0>0$, independent of $\varepsilon_0, N>0$, and positive numbers $(\varepsilon_k)_{k=1}^N$ satisfying the hypothesis $(i)-(iii)$. With these parameters at hand, we set 
\begin{equation}\label{EQ:Phi_k}
\phi_k(\zeta) = \sum_{I \in \Dy_{n_k}} \varepsilon_k \frac{\mu(I)}{|I|} 1_{I}(\zeta), \qquad \zeta \in \T,
\end{equation}
where $\Dy_{n_k}$ denotes a partition of $\T$ consisting of $n_k$ arcs of length $1/n_k$. Similarly to before, we also set 
\[
d\nu_k = \phi_k dm - \varepsilon_k d\mu , \qquad k=0,1,2,\dots
\]
and note that $\nu_k(I)=0$ for any arc $I \in \Dy_{n_k}$. This implies that for an arbitrary arc $J \subset \T$, we have that there are at most two arcs $I_1, I_2 \in \Dy_{n_k}$ such that
\[
\nu_k(J) =  \nu_k(I_1) + \nu_k(I_2) \leq \varepsilon_k n_k \mu(I_1) |J\cap I_1| +  \varepsilon_k n_k \mu(I_2) |J\cap I_2| \leq  2 \varepsilon_k.
\]
For the lower bound, the same argument gives  
\[
\nu_k(J) \geq - \varepsilon_k \mu(I_1) - \varepsilon_k \mu(I_2) \geq - 2\varepsilon_k.
\]
Hence we conclude that $\abs{\nu_k(J)} \leq 2 \varepsilon_k$ for any arc $J \subset \T$. As before, we shall transform this into the following growth estimate on the Poisson extensions of $\nu_k$.
\begin{lemma}\thlabel{LEM:Poissonest} There exists an absolute constant $c>0$, such that the following holds:
\[
P(\nu_k)(z) \leq c\varepsilon_k\min \left( n_k , \frac{1}{1-|z|}\right), \qquad z \in \D.
\]
\end{lemma}
The proof is principally similar to \thref{LEM:Poisest}, we omit the details.

\subsection{The main construction}
We now turn our attention to the main case.

\begin{proof}[Proof of \thref{THM:CIDD}]
According to the discussion in \thref{LEM:LINPROG}, we may assume that $W$ satisfies the condition $\lim_{t\to 0+} t\log W(t)=0$. 

\proofpart{1}{Uniform norm bound:}
Let $\varepsilon_0, N>0$ and $A>1$ be given. According to \thref{LEM:DISDINI}, we can pick positive integers $(n_k)_k$ satisfying the hypothesis therein. With this choice of $(n_k)_k$, we may apply \thref{LEM:LINPROG}, in order to obtain a constant $c_0>0$, independent on $\varepsilon_0,N>0$ and positive numbers $(\varepsilon_k)_{k=1}^N$, which satisfy the required properties $(i)-(iii)$ in the statement. We now form the corresponding bounded holomorphic functions defined by
\begin{equation}\label{EQ:DEFFN}
F_N(z) := \exp \left( \sum_{1\leq j \leq N} H(\phi_j)(z) \right), \qquad z\in \D,
\end{equation}
where $H$ denotes the Herglotz transform. Fix a number $\rho >1$ to be determined later and consider annuli's of the form $R_k := \{ 1/n_{k+1}<1-|z|\leq 1/n_k \} $. Invoking \thref{LEM:Poissonest} and $(ii)-(iii)$ of \thref{LEM:LINPROG}, we obtain the following estimate in $z\in R_k$:
\begin{multline*}
\abs{\Theta(z) F_N(z)} W(1-|z|)^{\rho} \leq  \exp \left( \sum_{1\leq j \leq N} P(\nu_j)(z) \right) W(1-|z|)^{\rho} \\
 \leq \exp \left( c_0\sum_{1\leq j \leq k} \varepsilon_j n_j \right) \cdot \exp \left( c_0 n_{k+1} \sum_{k< j\leq N} \varepsilon_j \right) W(1/n_k)^{\rho}
\lesssim W(1/n_k)^{(\rho-2c_0)}
\end{multline*}
Hence this quantity remains bounded if $\rho>2c_0$. Meanwhile, the estimate on the disc $|z| \leq 1-1/n_0$ is again carried out using \thref{LEM:Poissonest}:
\[
\sup_{|z|\leq 1-1/n_0} \abs{\Theta(z) F_N(z)} W(1-|z|)^{\rho} \leq \sup_{|z|\leq 1-1/n_0} \exp \left( \frac{c_0}{1-|z|} \sum_{k=1}^N \varepsilon_k  \right) \leq \exp ( c_0 n_0).
\]
As a consequence, there exists a constant $C>0$, independent of $N>0$, the initial value $\varepsilon_0>0$ of the sequence $(\varepsilon_j)_{j=0}^N$, and the minimal gap $A>1$ of the sequence $(n_j)_{j=0}^\infty$, such that
\[
\sup_{z \in \D} W(1-|z|)^{\rho} \abs{F_N(z) \Theta(z)} \leq C,
\]
whenever $\rho> 2c_0$.

\proofpart{2}{Passing to a convergent sequence:}
We now reintroduce the appropriate parameters, so that they all depend on the single parameter $N>0$. To this end, let $\varepsilon_0(N) \to 0$ and $A_N \to \infty$ as $N\to\infty$, and apply \thref{LEM:DISDINI} with parameter $A_N$ and \thref{LEM:LINPROG} with $\varepsilon_0(N)$, which give rise to positive integers $(n_j(N))_{j=0}^\infty$ and positive numbers $(\varepsilon_j(N))_{j=0}^N$. We now highlight the following crucial properties, needed for our purposes: 
\begin{enumerate}
    \item[(a.)] $n_{j+1}(N)\geq A_N n_j(N)$, $j=0,1,2,\dots$,
    \item[(b.)] $\sum_{j=0}^N \varepsilon_j(N)=1$,
    \item[(c.)] $\varepsilon_j(N) \leq c_0 \frac{1}{n_{j+1}(N)} \log \frac{1}{W(1/n_j(N))} \to 0$, as $N\to \infty$, for each fixed $j$,
\end{enumerate}
since the gaps $A_N \to \infty$, and $\kappa_W(t)\to 0$ as $t\to 0$. We shall now record the following lemma on weak-star convergence.

\begin{lemma}\thlabel{LEM:WEAKSTAR} Let $(F_N)_N$ be the functions defined as in \eqref{EQ:DEFFN}, where the corresponding parameters $(n_j(N))_{j=0}^\infty$ and $(\varepsilon_j(N))_{j=0}^N$, are defined as in the previous paragraph. Then the product $F_N \Theta$ converges to $1$ uniformly on compact subsets in $\D$.

\end{lemma}
\begin{proof}
Set 
\[
\phi_k dm -\varepsilon_k(N) d\mu = \varepsilon_k(N) \left( \sum_{I \in \Dy_{n_k}}  \frac{\mu(I)}{|I|} 1_{I}dm - d\mu \right)=: \varepsilon_k(N) d\sigma_k(N) , \qquad \zeta \in \T,
\]
and note that it clearly suffices to show that
\[
d\nu^N := \sum_{k=1}^N \phi_k dm - d\mu = \sum_{k=1}^N (\phi_k dm - \varepsilon_k(N)d\mu) =\sum_{k=1}^N \varepsilon_k(N) d\sigma_k(N),
\]
converges to zero in the weak-star topology of measures $M(\T)$, as $N\to \infty$. Fix an arbitrary $\eta>0$ and let $\psi$ be a continuous function on $\T$. By uniform continuity, we can find a large integer $M>1$, such that 
\begin{equation}\label{EQ:uniest}
\sup_{|\zeta-\xi|\leq 1/n_k(N) } \abs{\psi(\zeta)-\psi(\xi)} \leq \eta, \qquad n_k(N) > M.
\end{equation}
We now group the terms as follows:
\[
 \abs{\int_{\T} \psi d\sigma_k(N) } \leq \sum_{I \in \Dy_{n_k}} \abs{ \frac{\mu(I)}{|I|}\int_I \psi dm - \int_I \psi d\mu}.
\]
If $\xi_I$ denotes the center of each arc $I \in \Dy_{n_k}$, then
\[
\abs{ \frac{\mu(I)}{|I|}\int_I \psi dm - \int_I \psi d\mu} \leq \frac{\mu(I) }{|I|} \int_I \abs{\psi- \psi(\xi_I)} dm + \int_I \abs{\psi-\psi(\xi_I)} d\mu \leq 2 \eta \mu(I).
\]
Summing over all $I\in \Dy_{n_k}$, we get 
\[
\abs{\int_{\T} \psi d\sigma_k(N) } \leq \sum_{I \in \Dy_{n_k}} 2 \eta \mu(I) \leq 2 \eta \mu(\T), \qquad n_k(N) >M.
\]
On the other hand, we also have the trivial estimate 
\[
\abs{\int_{\T} \psi d\sigma_k(N) } \leq \norm{\psi}_\infty \norm{\sigma_k(N)} \leq 2 \norm{\psi}_\infty \mu(\T).
\]
Combining the above estimates, we obtain
\[
\abs{\int_{\T} \psi d\nu^N } \leq \sum_{k=1}^{N} \varepsilon_k(N) \abs{\int_{\T} \psi d\sigma_k(N) } \leq \sum_{k=1}^{M} \varepsilon_k(N) 2\norm{\psi}_{\infty}\mu(\T) + 2 \eta \mu(\T).
\]
Letting $N\to \infty$ finishes the proof.
\end{proof}

With this lemma at hand, we now invoke the compact embedding of $A^\infty(W^{\rho})$ into the space $A^\infty(W^M)$ for $M>\rho$, allowing us to pass a subsequence $F_{N_k}$ such that $F_{N_k}\Theta \to 1$ in $A^\infty(W^M)$. Since the $F_N$'s are zero-free in $\D$, we get that 
\begin{enumerate}
    \item[(i)] $\sup_k \norm{\Theta^{1/2M} F^{1/2M}_{N_k}}_{A^\infty(W^{1/2})} < \infty$,
    \item[(ii)] $\Theta^{1/2M}(z) F^{1/2M}_{N_k}(z) \to 1 $ uniformly on compact subsets of $\D$.
\end{enumerate}
Using the compact embedding of $A^\infty(W^{1/2}) \hookrightarrow A^\infty(W)$ from \thref{LEM:cptemb}, we conclude that $\Theta^{1/2M}$ is cyclic in $A^\infty(W)$ for large enough integers $M>1$. Since $\Theta$ is bounded, we may invoke \thref{LEM:bddcyc} to deduce that $\Theta$ is cyclic in $A^\infty(W)$. A similar argument also allows us to dispense the initial assumption that $\mu(\T)\leq 1$. The proof is now complete. \qedhere

\end{proof}

\newpage 
\bibliographystyle{siam}
\bibliography{mybib}

\Addresses

\end{document}